\documentclass[reqno, a4paper, 12pt]{amsart}

\usepackage{comment}

\usepackage[british]{babel}

\usepackage{bbm}
\usepackage[top=2.50cm,left=2.65cm,right=2.65cm,bottom=2.50cm]{geometry}
\usepackage{amsfonts}
\usepackage{amsmath,amsthm,amscd,mathrsfs,amssymb,graphicx,enumerate,
	dsfont}
\usepackage{xypic}
\usepackage{hyperref}
\usepackage[capitalise]{cleveref}
\usepackage[usenames,dvipsnames]{xcolor}
\hypersetup{colorlinks=true,citecolor=NavyBlue,linkcolor=Maroon,urlcolor=Orange}

\allowdisplaybreaks

\newtheorem{theorem}{Theorem}[section]
\newtheorem{lemma}[theorem]{Lemma}
\newtheorem{corollary}[theorem]{Corollary}
\newtheorem{conjecture}[theorem]{Conjecture}
\newtheorem{proposition}[theorem]{Proposition}

\theoremstyle{definition}
\newtheorem{remark}[theorem]{Remark}
\newtheorem{example}[theorem]{Example}
\newtheorem{definition}[theorem]{Definition}
\newtheorem{construction}{Construction}


\crefname{theorem}{Theorem}{Theorems}
\Crefname{theorem}{Theorem}{Theorems}
\crefname{conjecture}{Conjecture}{Conjectures}
\Crefname{conjecture}{Conjecture}{Conjectures}
\crefname{lemma}{Lemma}{Lemmas}
\Crefname{lemma}{Lemma}{Lemmas}
\crefname{corollary}{Corollary}{Corollaries}
\Crefname{corollary}{Corollary}{Corollaries}
\crefname{proposition}{Proposition}{Propositions}
\Crefname{proposition}{Proposition}{Propositions}
\crefname{definition}{Definition}{Definitions}
\Crefname{definition}{Definition}{Definitions}
\crefname{remark}{Remark}{Remarks}
\Crefname{remark}{Remark}{Remarks}
\crefname{example}{Example}{Examples}
\Crefname{example}{Example}{Examples}
\Crefname{equation}{\!}{\!}

\crefname{construction}{construction}{constructions}
\Crefname{construction}{Construction}{Constructions}

\crefformat{construction}{#2construction#3}
\Crefformat{construction}{#2Construction#3}

\let\oldconjecture\conjecture
\renewcommand{\conjecture}{%
  \crefalias{theorem}{conjecture}%
  \oldconjecture
}

\let\olddefinition\definition
\renewcommand{\definition}{%
  \crefalias{theorem}{definition}%
  \olddefinition
}

\let\oldlemma\lemma
\renewcommand{\lemma}{%
  \crefalias{theorem}{lemma}%
  \oldlemma
}

\let\oldcorollary\corollary
\renewcommand{\corollary}{%
  \crefalias{theorem}{corollary}%
  \oldcorollary
}

\let\oldproposition\proposition
\renewcommand{\proposition}{%
  \crefalias{theorem}{proposition}%
  \oldproposition
}

\let\oldremark\remark
\renewcommand{\remark}{%
  \crefalias{theorem}{remark}%
  \oldremark
}

\let\oldexample\example
\renewcommand{\example}{%
  \crefalias{theorem}{example}%
  \oldexample
}

\usepackage{amsmath,tikz-cd}
\usepackage{setspace}
\usepackage{amsfonts, mathrsfs}
\usepackage{amssymb}
\usepackage{mathtools}
\usepackage{enumitem}

\usepackage[utf8]{inputenc}

\usepackage{csquotes}
\usepackage[style=alphabetic, backend=biber, maxalphanames=5, maxbibnames=99, minbibnames=99, url=false]{biblatex}
\AtEveryBibitem{%
  \clearfield{issn}
}

\newcommand{\IN}{\ensuremath{\mathbb{N}}}
\newcommand{\IZ}{\ensuremath{\mathbb{Z}}}
\newcommand{\IQ}{\ensuremath{\mathbb{Q}}}

\newcommand{\IC}{\ensuremath{\mathbb{C}}}
\newcommand{\IP}{\ensuremath{\mathbb{P}}}

\newcommand{\id}{\mathrm{id}}

\DeclareMathOperator{\Ext}{Ext}

\newcommand{\Ox}{\ensuremath{\mathcal{O}}}
\newcommand{\E}{\ensuremath{\mathcal{E}}}
\newcommand{\F}{\ensuremath{\mathcal{F}}}

\renewcommand{\L}{\ensuremath{\mathcal{L}}}

\DeclareMathOperator{\Hom}{Hom}

\renewcommand{\comment}[1]{}

\DeclareMathOperator{\spec}{Spec}

\DeclareMathOperator{\End}{End}

\DeclareMathOperator{\chow}{CH}
\DeclareMathOperator{\gdch}{GDCH}

\DeclareMathOperator{\sym}{Sym}
\DeclareMathOperator{\pr}{pr}

\newcommand*{\intref}[2]{\def\tmp{#2}\ifx\tmp\empty\hyperref[#1]{\ref*{#1}}\else\hyperref[#1]{#2~\ref*{#1}}\fi}

\newcommand*{\intrefeq}[2]{\def\tmp{#2}\ifx\tmp\empty\hyperref[#1]{(\ref*{#1})}\else\hyperref[#1]{#2~(\ref*{#1})}\fi}

\newcommand*{\mainref}[2]{\def\tmp{#2}\ifx\tmp\empty\hyperref[#1]{\ref*{#1}}\else\hyperref[#1]{#2}\fi}

\begin{document}

\title[Zero-cycles on moduli spaces of twisted sheaves]{Zero-cycles on moduli spaces of twisted sheaves and applications to double EPW quartics}
\author{Carl Mazzanti}

\begin{abstract}
    Chen, Li, Zhang, and Zhang extended the results of Shen, Yin, and Zhao on zero-cycles on moduli spaces of stable objects on $K3$ surfaces to the twisted setting.
    In this work, we complement this by extending results by Vial and Martin--Vial to moduli spaces on twisted $K3$ surfaces.
    Exploiting the fact that double EPW quartics can be realised as moduli spaces of twisted sheaves, we show that effective zero-cycles agree if and only if they agree in the associated Verra fourfold and show that the twisted Beauville--Voisin class of Chen, Li, Zhang, and Zhang agrees with the Beauville--Voisin class in that case.
\end{abstract}

\maketitle

\tableofcontents

\section{Introduction}
\subsection{\texorpdfstring{Zero-cycles on moduli spaces of twisted sheaves on $K3$ surfaces}{Zero-cycles on moduli spaces of twisted sheaves on K3 surfaces}}

Let $S$ be a $K3$ surface and $o_S\in\chow_0(S)$ its Beauville--Voisin class in the Chow group of zero-cycles with rational coefficients, defined as the class of any point lying on a rational curve in $S$ \cite{BV}.
O'Grady considered a filtration on the zero-cycles of $S$ by defining~$S_d^{\text{OG}}\chow_0(S)\subset \chow_0(S)$ as the subset of all zero-cycles of the form $[p_1]+\ldots+[p_d]+ao_S$ for $a\in\IQ$ and points $p_i\in S$. 
O'Grady conjectured that
$$ c_2(\E)\in S_{d(\E)}^{\mathrm{OG}}\chow_0(S)$$
for any $\E\in D^b(S)$, where $d(\E)=\frac{1}{2}\dim\Ext^1(\E,\E)$.

This was verified by Shen, Yin, and Zhao, who defined a filtration on zero-cycles of moduli spaces of stable objects $\mathcal{M}_\sigma(S,v)$ in \cite{shen-yin-zhao} as 
$$
S_i^{\text{SYZ}}\chow_0(\mathcal{M}_\sigma(S,v))\coloneq \left\langle [\mathcal{E}]~|~\mathcal{E}\in\mathcal{M}_\sigma(v)\text{ and }c_2(\mathcal{E})\in S_i^{\text{OG}}\chow_0(S)\right\rangle.
$$
This was shown in \emph{loc. cit.} to be independent of the modular interpretation and by Li and Zhang \cite{li-zhang} to agree with a filtration defined by Voisin \cite{voisin16} in the more general context of hyperkähler varieties.
Following Shen, Yin, and Zhao's work, Barros, Flapan, Marian, and Silversmith introduced another such filtration in \cite{bfms}.
This was subsequently generalised to settings much more general than moduli spaces of stable objects on $K3$ surfaces and called the co-radical filtration by Vial \cite{vial22}. 
All of these filtrations are known to agree \cite{vial22,li-zhang}.
The results of \cite{shen-yin-zhao} have recently been extended to the setting of twisted $K3$ surfaces by Chen, Li, Zhang, and Zhang \cite{twistedSYZ}.

We extend the fact that Vial's co-radical filtration agrees with the Shen--Yin--Zhao and Voisin's filtration to the setting of twisted $K3$ surfaces.

\begin{theorem}[\intref{prop:twisted-filtrations}{Proposition}]
    Let $\mathscr{X}=(S,\beta)$ be a twisted $K3$ surface and $\mathcal{M}_\sigma(\mathscr{X},v)$ a moduli space of stable objects in the derived category of twisted sheaves $D^{(1)}(\mathscr{X})$. 
    Then Voisin's, the Shen--Yin--Zhao, and Vial's co-radical filtrations agree,
    $$
    S_i^{\text{V}}\chow_0(\mathcal{M}_\sigma(\mathscr{X},v))=S_i^{\text{SYZ}}\chow_0(\mathcal{M}_\sigma(\mathscr{X},v))=R_i\chow_0(\mathcal{M}_\sigma(\mathscr{X},v)).
    $$
\end{theorem}

Work by Marian and Zhao \cite{marian-zhao} established that
$$[\E]=[\F]\text{ in }\chow_0(\mathcal{M}_\sigma(S,v))\iff c_2(\E)=c_2(\F)\text{ in }\chow_0(S)$$
and Martin and Vial \cite{martin-vial} generalised this to higher-degree effective zero-cycles. 
We extend this result to the twisted setting.

\begin{proposition}[\intref{prop:criterion1}{Proposition}]\label{intro:criterion}
    Consider a twisted $K3$ surface $\mathscr{X}=(S,\beta)$ and let~$\mathcal{M}_\sigma(\mathscr{X},v)$ be a moduli space of stable objects in $D^{(1)}(\mathscr{X})$ of dimension $2n$.
    Let $\E_i, \F_i$ be points in $\mathcal{M}_\sigma(\mathscr{X},v)$.
    Then
    \begin{multline*}
    \sum_{i=1}^m[\E_i]=\sum_{i=1}^m[\F_i]\text{ in }\chow_0(\mathcal{M}_\sigma(\mathscr{X},v))\\
    \iff \sum_{i=1}^m c_2(\E_i)^{\times k}=\sum_{i=1}^m c_2(\F_i)^{\times k} \text{ in }\chow_0(S^k) \text{ for all }k\leq\min(m,n).
    \end{multline*}
\end{proposition}

\subsection{The twisted Beauville--Voisin class}

Chen, Li, Zhang, and Zhang introduced a \emph{twisted Beauville--Voisin class} in the setting of twisted $K3$ surfaces $\mathscr{X}=(S,\beta)$ \cite{twistedSYZ}.
It is defined as 
$$
o_{\mathscr{X}}=\frac{c_2(\F)}{\mathrm{rk}(\F)}+\left(1-\frac{\deg c_2(\F)}{\mathrm{rk}(\F)} \right) o_S,
$$
where $\F$ is a locally free sheaf lying on a constant cycle subvariety of some $\mathcal{M}_\sigma(\mathscr{X},v)$.
It is shown to depend only on $\mathscr{X}$ in \cite{twistedSYZ} and is a main technical tool in \emph{loc. cit.}, required to adapt the results of \cite{shen-yin-zhao} to the twisted setting.
It is conjectured that~$o_\mathscr{X}=o_S$.
We provide evidence for this via the Franchetta conjecture and prove it in a special case.

\begin{theorem}[\intref{thm:BV-classes}{Theorem}]
    Let $S$ be a $K3$ surface of genus $2$ and $\beta\in\mathrm{Br}(S)$ the Brauer class of order $2$ discussed in \cite[Sec. 9.8]{van-geemen}. 
    Let $\mathscr{X}=(S,\beta)$ be the associated twisted $K3$ surface.
    Then the twisted and untwisted Beauville--Voisin classes agree, i.e.
    $$o_{\mathscr{X}}=o_S.$$
\end{theorem}
The proof of this theorem goes through a more geometric description of the moduli spaces $\mathcal{M}_{(0,H,0)}(\mathscr{X})$ as double EPW quartics \cite{ikkr17}.
This allows for more control over the constant cycle subvarieties of the moduli space and hence of the second Chern class of points lying on them.

\subsection{Double EPW quartics as moduli spaces of twisted sheaves}

Double EPW quartics $Y$ are hyperkähler fourfolds of $K3^{[2]}$-type first introduced by Iliev, Kapustka, Kapustka, and Ranestad in \cite{ikkr17}. 
They admit an anti-symplectic involution $\iota$ and the general $Y$ is a moduli space of stable objects on a twisted $K3$ surface of genus~$2$. 
This isomorphism was first obtained via lattice-theoretic methods in \cite{ikkr17} and later made more explicit in \cite{CamereKapustkaMongardi2018}.
Double EPW quartics also admit a more geometric description, similar to double EPW sextics, via conics in Verra fourfolds. 
We exploit this description to study the group of zero-cycles on $Y$ and relate it to the underlying~$K3$ surface.
The involution induces an eigenspace decomposition
$$\chow_0(Y)=\IQ o_Y\oplus\chow_0(Y)^-\oplus\chow_0(Y)^+_{\mathrm{hom}},$$
where $o_Y$ is the class of any point fixed by $\iota$.

\begin{theorem}[\intref{thm:K3-iso}{Theorem}]
    Let $Y$ be a general double EPW quartic and let $S$ be either of the two K3 surfaces associated to $Y$. 
    Then the map
    $$\chow_0(Y)\cong \chow_0(\mathcal{M}_{(0,H,0)}(S,\beta))\xrightarrow{\mathrm{ch_2}} \chow_0(S)$$
    is $0$ on $\chow_0(Y)^+_{\mathrm{hom}}$, maps $o_Y$ to $o_S$, and induces an isomorphism 
    $$\chow_0(Y)^-\cong \chow_0(S)_{\mathrm{hom}}.$$
    Its inverse is given up to scalars by the correspondence $\mathrm{ch}_2^t$ composed with multiplication by the square of the polarisation $h^2$. 
\end{theorem}

We introduce another, more geometric filtration on $\chow_0(Y)$, similar to the one introduced by Shen and Yin \cite{shen-yin} in the setting of Fano varieties of lines and show that it agrees with the filtrations discussed above and with the one induced by $\iota$.
This result provides another proof of \cite[Prop. 5.14]{dEPW-quartics-mazzanti}.
Similar results are known for Fano varieties of lines \cite{shenvial,voisin16} and double EPW sextics \cite{bolognesi-laterveer}.

\begin{theorem}[\intref{prop:filtrations}{Proposition}]
    Let $Y$ be a general double EPW quartic.
    Then, all four filtrations, the Shen--Yin, Voisin's, Shen--Yin--Zhao, and Vial's co-radical filtration, agree,
    $$
    S_i^{\text{SY}}\chow_0(Y)=S_i^{\text{V}}\chow_0(Y)=S_i^{\text{SYZ}}\chow_0(Y)=R_i\chow_0(Y).
    $$
    Furthermore, they all agree with the filtration obtained from the decomposition of $\chow_0(Y)$,
    $$
    \IQ o_Y\subset \IQ o_Y\oplus \chow_0(Y)^-\subset\chow_0(Y).
    $$
\end{theorem}

We also prove a criterion analogous to the one from \intref{intro:criterion}{Proposition} in the setting of double EPW quartics that takes the geometric construction via conics in Verra fourfolds into account. 
There is an analogous result in the case of the Fano variety of lines on a cubic fourfold \cite[Thm. 2.8]{martin-vial}.

\begin{theorem}[\intref{thm:criterion2}{Theorem}]
    Let $Y$ be a general double EPW quartic with associated Verra fourfold $X$ and $P$ the universal conic.
    Let $x_i,y_i\in Y$ and $c_i, c_i'$ be $(1,1)$-conics in~$X$ satisfying $P_\ast[c_i]=[x_i]$ and $P_\ast[c_i']=[y_i]$ in $\chow_0(Y)$.
    Then
    \begin{equation*}
        \sum_{i=1}^m[x_i]=\sum_{i=1}^m[y_i]\text{ in }\chow_0(Y)\iff 
        \begin{cases}
            \sum_{i=1}^m [c_i]=\sum_{i=1}^m [c_i']\text{ in }\chow_1(X), \text{ and}\\
            \sum_{i=1}^m[c_i]\times[c_i]=\sum_{i=1}^m[c_i']\times [c_i']\text{ in }\chow_2(X\times X).
        \end{cases}
    \end{equation*}
\end{theorem}
    
\subsection{Structure}

In \intref{sec:birat}{Section}, we define birational motives and discuss results that will be necessary in later sections.
In \intref{sec:twisted}{Section} we discuss zero-cycles on moduli spaces of stable objects on $K3$ surfaces and prove general results on zero-cycles in the twisted case.
In \intref{sec:quartics}{Section} we introduce double EPW quartics as moduli spaces of twisted sheaves on~$K3$ surfaces and study the consequences of the general results.

\subsection{Conventions}

We work over the complex numbers.
A variety refers to a separated, integral scheme of finite type over $\IC$ and a subvariety to a closed, reduced, possibly reducible, subscheme of a variety, unless stated otherwise. 
A hyperkähler variety refers to a smooth, projective variety that is simply connected in the analytic topology and carries a holomorphic symplectic form that is unique up to scalars, i.e. $H^0(T,\Omega_T^2)=\IC\sigma$.
Points in varieties refer to closed points, unless otherwise specified.
Chow rings will be assumed to have rational coefficients and $\chow^\ast(T)_{\mathrm{hom}}$ always refers to homologically trivial cycles with respect to singular cohomology.
A general point in a variety refers to a point contained in a non-empty open subset, while a very general point in a variety refers to a point contained in the non-empty complement of a countable union of subvarieties. 

\subsection{Acknowledgements}

The author thanks Charles Vial for initiating this project and for invaluable conversations over the course of completing this work.
The author is funded by the Deutsche Forschungsgemeinschaft (DFG, German Research Foundation) – Project-ID 491392403 – TRR 358.

\section{Birational motives}\label{sec:birat}

We begin by providing some motivation.
The category of birational motives $\mathcal{M}^\circ(\IC)_{\IQ}$ has the property that $\Hom_{\mathcal{M}^\circ(\IC)_\IQ}(\mathfrak{h}^\circ(Y),\mathfrak{h}^\circ(X))=\chow_0(X_{\IC(Y)})$ for two smooth projective varieties $X$ and $Y$.
Hence, by resolution of singularities, $\mathfrak{h}^\circ(X)$ encodes information about $\chow_0(X_L)$ for all finitely generated field extensions $L/\IC$.
Furthermore, $\chow_0$ is not a tensor functor, as the natural map $\chow_0(X)\otimes\chow_0(Y)\to\chow_0(X\times Y)$ is not an isomorphism in general.
Birational motives, on the other hand, do form a tensor functor and are easier to handle than Chow motives, making them the ideal setting in which to discuss zero-cycles. 
Finally, they provide the natural framework and language to prove the results that follow, as avoiding them would be significantly more technical.

In this section we work over an arbitrary field $k$ for the sake of generality, as there is no advantage in working over $\IC$. 
We define birational motives and recall results that will be needed. 
We refer to \cite{kahn-sujatha} and \cite{vial22} for more details, the latter in particular with respect to the co-multiplicative structure.

\begin{definition}\label{def:eff-motives}
    The (covariant) category $\mathcal{M}^{\mathrm{eff}}(k)_R$ of \emph{effective Chow motives} over $k$ with coefficients in $R$ is defined as the pseudo-effective envelope of the category $\mathrm{Corr}(k)_R$, the objects of which are smooth projective varieties over $k$.
    Define the morphisms by~$\Hom_{\mathrm{Corr}(k)_R}(X,Y)\coloneq \chow^{\dim Y}(X\times_k Y)\otimes R$.
    The composition is given by the composition of correspondences, i.e. 
    $$\Gamma\circ\Gamma'\coloneq (\pr_{X,Z})_\ast(\pr_{X,Y}^\ast\Gamma'\cdot \pr_{Y,Z}^\ast\Gamma). $$
\end{definition}
 
Explicitly, objects of $\mathcal{M}^{\mathrm{eff}}(k)_R$ are given by pairs $(X,p)$, where $X$ is a smooth projective variety over $k$ and $p\in\chow^{\dim X}(X\times_k X)$ is an idempotent correspondence. 
A morphism $f\colon (X,p)\to (Y,q)$ is given by a correspondence $\Gamma\in\chow^{\dim X}(X\times_k Y)$ that satisfies~$q\circ\Gamma=\Gamma\circ p$. 
Define the Chow motive of $X$ as $\mathfrak{h}(X)\coloneq(X,\delta_X)$, where $\delta_X$ is the class of the graph of the diagonal. 
Define a tensor product by $(X,p)\otimes(Y,q)\coloneq(X\times_k Y, p\times q)$ with unit $\mathds{1}=\mathfrak{h}(\spec k)$.
This endows the category of effective Chow motives with the structure of an $R$-linear tensor category.

The motive of $\IP^1$ decomposes as $\mathfrak{h}(\IP^1)=(\IP^1,\IP^1\times \{0\})\oplus (\IP^1,\{0\}\times\IP^1)=\mathds{1}\oplus\mathbb{L}$, the latter summand of which is called the \emph{Lefschetz motive}. 

\begin{definition}\label{def:chow-motives}
    The category $\mathcal{M}(k)_R$ of \emph{Chow motives} over $k$ with coefficients in $R$ is constructed from $\mathcal{M}^{\mathrm{eff}}(k)_R$ by inverting the functor $-\otimes\mathbb{L}$. 
    It is rigid and the natural functor $\mathcal{M}^{\mathrm{eff}}(k)_R\to\mathcal{M}(k)_R$ fully functor. 
    We define $\mathfrak{h}(X)(n)$ as $\mathfrak{h}(X)\otimes(\mathbb{L}^{-1})^{\otimes n}$. 
\end{definition}

Instead of inverting $-\otimes\mathbb{L}$ one can take the quotient by it to obtain the category of \emph{birational Chow motives} over $k$ with coefficients in $R$, denoted by $\mathcal{M}^\circ(k)_R$, as introduced by Kahn and Sujatha in \cite{kahn-sujatha}.

\begin{definition}\label{def:birat-motives}
    Let $\mathcal{L}$ be the $\otimes$-ideal in $\mathcal{M}^{\mathrm{eff}}(k)_R$ consisting of morphisms which factor through an object of the form $P\otimes\mathbb{L}$.
    Define the category $\mathcal{M}^\circ(k)_R$ of birational motives as the pseudo-abelian envelope of $\mathcal{M}^{\mathrm{eff}}(k)_R/\mathcal{L}$. 
    Denote the birational motive of $X$ by~$\mathfrak{h}^\circ(X)$.
\end{definition}

Morphisms of birational motives also admit an explicit description. 
For irreducible~$X$, we have 
$$
\Hom(\mathfrak{h}^\circ(X),\mathfrak{h}^\circ(Y))\coloneq \chow_0(Y_{k(X)})\otimes R=\varinjlim_{U\subset X} \chow_{\dim X}(U\times_k Y)\otimes R,
$$
where the limit runs over all non-empty open subsets of $X$. 
A morphism in $\mathcal{M}^{\mathrm{eff}}(k)_R$ induces a morphism in $\mathcal{M}^\circ(k)_R$ by restricting the correspondence to the generic point. 
Furthermore, a rational map $f\colon X\dashrightarrow Y$ induces a morphism $f_\ast\colon \mathfrak{h}^\circ(X)\to\mathfrak{h}^\circ(Y)$ by restricting the graph of $f|_U$ to the generic point, where $U$ is a dense open on which~$f$ is defined. 
Additionally, a generically finite rational map $f\colon X\dashrightarrow Y$ induces a well-defined morphism $f^\ast\colon \mathfrak{h}^\circ(Y)\to\mathfrak{h}^\circ(X)$ by \cite[Cor. 2.5]{shen-birat}.

A key feature of morphisms of birational motives with values in $\IQ$ is that they are defined by their action on zero-cycles after base-changing to a sufficiently large field.

\begin{lemma}[{\cite[Lem. 2.1]{vial22}}]\label{lemma:defined-by-chow}
    Let $X$ and $Y$ be two smooth projective varieties over a field $k$ and $\Omega$ a universal domain containing $k$. 
    Let $\gamma$ and $\gamma'$ be two morphisms~$\mathfrak{h}(X)\to~\mathfrak{h}(Y)$ of Chow motives. 
    Then 
    $$
    \gamma_\ast = \gamma'_\ast\colon \chow_0(X_\Omega)\otimes\IQ\to\chow_0(Y_\Omega)\otimes\IQ \iff \gamma=\gamma'\colon\mathfrak{h}^\circ(X)_\IQ\to\mathfrak{h}^\circ(Y)_\IQ.
    $$
\end{lemma}

The pullback along the diagonal morphism $\delta\colon X\to X\times_k X$ endows the contravariant motive $\mathfrak{h}(X)$ with a commutative algebra structure with unit induced by the structure morphism $ X\to \spec k$.
Working covariantly instead, the push-forward along the diagonal endows $\mathfrak{h}(X)$ with the structure of a co-commutative co-algebra. 
Passing to birational motives, we obtain a co-commutative co-algebra structure on $\mathfrak{h}^\circ(X)$ with co-unit $\epsilon\colon \mathfrak{h}^\circ(X)\to \mathds{1}$ induced by pushing forward along the structure morphism. 
It is straightforward to check that the identities necessary to ensure that this indeed defines a co-commutative co-algebra hold at the level of morphisms of varieties already. 
Finally, just as a morphism of varieties $f\colon X\to Y$ induces a morphism $f^\ast\colon \mathfrak{h}(Y)\to\mathfrak{h}(X)$ of algebra objects, a rational morphism $f\colon X\dashrightarrow Y$ induces a morphism $f_\ast\colon\mathfrak{h}(X)\to\mathfrak{h}(Y)$ of co-algebra objects.
We recollect the following lemma about morphisms of birational motives as co-algebra objects, as it will be needed later.

\begin{proposition}[{\cite[Prop. 2.3]{vial22}}]\label{lemma:co-algebra-iso}
    Let $X$ and $Y$ be smooth varieties over $k$ of dimension $d$ and let $\Omega$ be a universal domain containing $k$. 
    Assume there exists a projective variety $\Gamma$ of dimension $d$ and generically finite morphisms $\phi\colon\Gamma\to X$ and~$\psi\colon\Gamma\to~Y$, such that one of the following equivalent conditions hold.
    \begin{enumerate}
        \item $\phi_\ast[p]=\phi_\ast[q]$ in $\chow_0(X_\Omega)$ for any two general points $p$ and $q$ in $\Gamma(\Omega)$ that lie in the same fibre of $\psi$.
        \item $\phi_\ast\psi^\ast\psi_\ast\alpha=\deg(\psi)\phi_\ast\alpha$ for any $\alpha\in\chow_0(\Gamma_\Omega)$.
    \end{enumerate}
    Then
    \begin{enumerate}[label=\alph*)]
        \item $\gamma\coloneq \frac{1}{\deg\phi}\psi_\ast\phi^\ast\colon\mathfrak{h}^\circ(X)\to\mathfrak{h}^\circ(Y)$ is split injective,
        \item $\gamma'\coloneq \frac{1}{\deg \psi}\phi_\ast\psi^\ast\colon\mathfrak{h}^\circ(Y)\to\mathfrak{h}^\circ(X)$ is split surjective and $\gamma'\circ\gamma=\id_{\mathfrak{h}^\circ(X)}$
        \item the following diagram commutes.
        $$\xymatrix{
        \mathfrak{h}^\circ(X)\ar[d]^\gamma \ar[rr]^{\delta_X} & & \mathfrak{h}^\circ(X)\otimes\mathfrak{h}^\circ(X)\\
        \mathfrak{h}^\circ(Y)\ar[rr]^{\delta_Y}& & \mathfrak{h}^\circ(Y)\otimes\mathfrak{h}^\circ(Y)\ar[u]^{\gamma'\otimes\gamma'}
        }$$
    \end{enumerate}
    In particular, if $\psi_\ast[p]=\psi_\ast[q]$ for any two general points $p,q\in\Gamma(\Omega)$ lying on the same fibre of $\phi$, then $\gamma\colon\mathfrak{h}^\circ(X)\to\mathfrak{h}^\circ(Y)$ is an isomorphism of co-algebra objects with inverse~$\gamma'$.
\end{proposition}

We also define (co-multiplicative) birational Chow--Künneth decompositions, as they will be needed later. 
They are the counterpart to (multiplicative) Chow--Künneth decompositions in the setting of birational motives.
We refer to \cite[Sec. 4]{vial22} for more details.

\begin{definition}
    Fix a Weil cohomology theory $H^\ast$ for smooth projective varieties defined over $k$.
    Define the transcendental cohomology of a smooth projective variety as 
    $$ H_{\mathrm{tr}}^k(X)\coloneq H^k(X)/N^1 H^k(X),$$
    where $N^\bullet$ denotes the coniveau filtration
    $$ N^rH^k(X)\coloneq \sum_{Z\subset X} \ker(H^k(X)\to H^k(X\setminus Z)),$$
    and the sum runs over all codimension $r$ subvarieties of $X$.
\end{definition}

A morphism $$\gamma\in \Hom(\mathfrak{h}^\circ(X),\mathfrak{h}^\circ(Y))=\chow_0(Y_{\IC(X)})$$ induces a homomorphism 
$$\gamma^\ast\colon H^k_{\mathrm{tr}}(Y)\to H^k_{\mathrm{tr}}(X),$$
by choosing a lift of $\gamma$ to $\chow^{\dim X}(X\times Y)$ that acts on $H^k(Y)$. 
This is well-defined, as the difference between any two lifts will be supported on $D\times Y$ for some divisor $D\subset X$ and hence maps $H^k(Y)$ to $N^1H^k(X)$.

\begin{definition}[\cite{shen-birat}]\label{def:birational-CK}
    Let $X$ be a smooth projective variety over $k$.
    A \emph{birational Chow--Künneth decomposition} is a collection $\{\varpi_0,\ldots, \varpi_d\}\subset \End(\mathfrak{h}^\circ(X))$ such that
    \begin{enumerate}
        \item $\id_{\mathfrak{h}^\circ(X)}=\varpi_0+\ldots+\varpi_d$
        \item $\varpi_k\circ\varpi_k=\varpi_k$
        \item $\varpi_i\circ\varpi_j=0$ for $i\neq j$
        \item $\varpi_k^\ast\colon H^l_{\mathrm{tr}}\to H^l_{\mathrm{tr}}$ is the identity if $l=k$ and $0$ otherwise.
    \end{enumerate}
    Following Vial \cite{vial22}, we say that a birational Chow--Künneth decomposition is \emph{co-multiplicative} if the induced decomposition defines a unital grading of the co-algebra object $\mathfrak{h}^\circ(X)$, or equivalently if
    \begin{enumerate}
        \item $\varpi_0=o_{k(X)}$ for some $o\in\chow_0(X)$, and
        \item $(\varpi_i\otimes\varpi_j)\circ\delta_X\circ\varpi_k=0$ if $k\neq i+j$.
    \end{enumerate}
\end{definition}

\section{\texorpdfstring{Zero-cycles on moduli spaces of twisted sheaves on $K3$ surfaces}{Zero-cycles on moduli spaces of twisted sheaves on K3 surfaces}}\label{sec:twisted}

We begin by reviewing what is known in the case of moduli spaces of stable objects on $K3$ surfaces before discussing what changes have to be made in the twisted case. 

Let $S$ be a $K3$ surface and denote its bounded derived category of coherent sheaves by $D^b(S)$. 
Let $\mathcal{M}_\sigma(S, v)$ be a moduli space of stable objects in $D^b(S)$ with respect to $\sigma$ and of fixed primitive Mukai vector $v$.
It is a hyperkähler variety of dimension $\frac{v^2+2}{2}$.

Following the notation in \cite{twistedSYZ}, we denote a twisted $K3$ surface by $\mathscr{X}=(S,\beta)$, where $\beta$ is a Brauer class on $S$. 
Denote the category of $d$-fold $\beta$-twisted sheaves on $S$ by $D^{(d)}(\mathscr{X})$.
We denote a moduli space of stable objects on $\mathscr{X}$ by $\mathcal{M}_\sigma(\mathscr{X},v)$. 
See \cite[Sec. 2]{twistedSYZ} for more details.

We may abbreviate the notation $\mathcal{M}_\sigma(S,v)$ and $\mathcal{M}_\sigma(\mathscr{X},v)$ to $\mathcal{M}$ when there is no danger of confusion.

\subsection{The untwisted case}\label{sec:untwisted}

Beauville and Voisin proved in \cite{BV} the existence of a special zero-cycle of $o_S\in\chow_0(S)$, which is defined as the class of any point lying on a rational curve in $S$.
It has the property that the second Chern class and all products of divisors are a multiple of $o_S$.
O'Grady defined the following increasing filtration by subsets of $\chow_0(S)$ in \cite{og2}.
\begin{definition}[\cite{og2}]\label{def:og}
    Define \emph{O'Grady's filtration} on $\chow_0(S)$ as 
    $$
    S_i^{\text{OG}}\chow_0(S)=\bigcup_{k\in \IQ, z_j\in S}\{ko_S+[z_1]+\ldots [z_i] \}
    $$
\end{definition}

After having first been conjectured to be the case by O'Grady in \cite{og2}, the following was proven by Shen, Yin, and Zhao.

\begin{theorem}[\cite{shen-yin-zhao}]
    For any object $\mathcal{E}\in D^b(S)$, we have
    $$ c_2(\mathcal{E})\in S^{\text{OG}}_{d(\mathcal{E})}\chow_0(S),$$
    where $d(\mathcal{E})=\frac{1}{2}\dim\Ext^1(\mathcal{E},\mathcal{E})$.
\end{theorem}

This result is necessary to ensure that the following filtration, induced by O'Grady's filtration, is exhaustive.

\begin{definition}[\cite{shen-yin-zhao}]\label{def:untwisted-syz}
    Let $S$ be a $K3$ surface and $\mathcal{M}_\sigma(S,v)$ a moduli space of stable objects in $D^b(S)$. 
    Define the \emph{Shen--Yin--Zhao filtration} as
    $$
    S_i^{\text{SYZ}}\chow_0(\mathcal{M}_\sigma(S,v))\coloneq \left\langle [\mathcal{E}]~|~\mathcal{E}\in\mathcal{M}_\sigma(S,v)\text{ and }c_2(\mathcal{E})\in S_i^{\text{OG}}\chow_0(S)\right\rangle.
    $$
\end{definition}

In a quest to find a splitting of the conjectural Bloch--Beilinson filtration on the Chow ring \cite{voisin16}, Voisin defined a filtration on zero-cycles of hyperkähler varieties that is expected to split the conjectured filtration. 
Barros, Flapan, Marian, and Silversmith considered a different filtration in the setting of moduli spaces of stable objects on $K3$ surfaces \cite{bfms}, which was subsequently generalised to other varieties by Vial \cite{vial22}.
These filtrations have all been shown to agree in the case of moduli spaces of sheaves.

\begin{definition}[\cite{voisin16}]\label{def:voisin-filtration}
    Let $T$ be a hyperkähler variety of dimension $n$ and let $O_x$ be the set of points rationally equivalent to $x\in T$.
    This is a countable union of closed algebraic subvarieties of $T$ and has a well-defined dimension as the supremum over the dimensions of these subvarieties.
    Define \emph{Voisin's filtration} on $\chow_0(T)$ as
    $$
    S_i^{\text{V}}\chow_0(T)\coloneq \langle[x]~|~x\in T\text{ and }\dim O_x\geq n-i\rangle
    $$
\end{definition}

\begin{definition}[\cite{vial22}]\label{def:co-radical}
    Let $T$ be a smooth projective variety and fix a \emph{point-like} zero-cycle $o\in\chow_0(T)$, meaning that $(\Delta_T)_\ast o = o\times o$.
    Define $\overline{\delta}^0=\id-o\cdot \deg$, $\overline{\delta}=(\Delta_T-o\times\id-\id\times o)\circ\overline{\delta}^0$, and $\overline{\delta}^k=(\overline{\delta}\otimes\id\otimes\ldots\otimes\id)\circ\overline{\delta}^{k-1}$. 
    \emph{Vial's co-radical filtration} is defined as
    $$
    R_k(T)\coloneq \ker(\overline{\delta}^k_\ast\colon\chow_0(T)\to \chow_0(T^{k+1}))
    $$
\end{definition}

\begin{theorem}[\cite{li-zhang},\cite{vial22}]\label{thm:untwisted-filtrations}
    Let $S$ be a $K3$ surface and $\mathcal{M}_\sigma(S,v)$ a moduli space of stable objects in $D^b(S)$.
    Then Voisin's, the Shen--Yin--Zhao, Vial's co-radical, and the BFMS filtration agree, i.e.
    $$ S_i^{\text{V}}\chow_0(\mathcal{M})=S_i^{\text{SYZ}}\chow_0(\mathcal{M})=S_i^{\text{BFMS}}\chow_0(\mathcal{M})=R_i\chow_0(\mathcal{\mathcal{M}}).$$
\end{theorem}

These results rely on the fact that 
$$
[\mathcal{E}]=[\mathcal{F}]\text{ in }\chow_0(\mathcal{M}_\sigma(v))\iff c_2(\mathcal{E})= c_2(\mathcal{F})\text{ in }\chow_0(S),
$$
which was first conjectured in \cite{shen-yin-zhao} and subsequently proven in \cite{marian-zhao}.
This led Voisin to conjecture that two points $x$ and $y$ on a hyperkähler variety $T$ define the same zero-cycle if and only if $[x]=[y]$ in $\chow_0(T)/F^3$, where $F^\bullet$ is the conjectural Bloch--Beilinson filtration, see \cite{voisin22b}.
This was extended to the following conjecture in \cite{martin-vial}.

\begin{conjecture}
    Let $T$ be a smooth projective variety whose algebra of holomorphic forms is generated in degree $\leq d$. Then 
    $$\sum_{i=1}^m[x_i]=\sum_{i=1}^m[y_i]\text{ in }\chow_0(T)\text{ if and only if }\sum_{i=1}^m[x_i]=\sum_{i=1}^m[y_i]\text{ in }\chow_0(T)/F^{md+1}\chow_0(T),$$
    for the conjectural Bloch--Beilinson filtration $F^\bullet$ and any $x_1,\ldots, x_m,y_1,\ldots, y_m\in T$.
\end{conjecture}

This has been verified in \cite{martin-vial} for moduli spaces of stable objects on $K3$ surfaces, generalised Kummer varieties, Fano varieties of lines on a cubic fourfold, and O'Grady's $6$-dimensional hyperkähler varieties for an appropriate candidate for the Bloch--Beilinson filtration. 
In particular, Martin and Vial obtained the following explicit criterion.

\begin{proposition}[\cite{martin-vial}]\label{prop:untwisted-criterion}
    Let $S$ be a $K3$ surface and $\mathcal{M}_\sigma(S,v)$ a moduli space of stable objects in $D^b(S)$ of dimension $2n$.
    Let $\mathcal{E}_1,\ldots \mathcal{E}_m,\mathcal{F}_1,\ldots\mathcal{F}_m$ be points in $\mathcal{M}_\sigma(S,v)$.
    Then
    \begin{multline*}
    \sum_{i=m}^m[\E_i]=\sum_{i=1}^m[\F_i]\text{ in }\chow_0(\mathcal{M}_\sigma(S,v))\\
    \iff \sum_{i=1}^m c_2(\E_i)^{\times k}=\sum_{i=1}^m c_2(\F_i)^{\times k} \text{ in }\chow_0(S^k) \text{ for all }k\leq\min(m,n).
    \end{multline*}
\end{proposition}

\subsection{The twisted case}

A \emph{twisted Beauville--Voisin class} $o_{\mathscr{X}}\in \chow_0(S)$ was introduced by Chen, Li, Zhang, and Zhang \cite{twistedSYZ} for a twisted $K3$ surface $\mathscr{X}=(S,\beta)$.
It is defined as 
$$
o_{\mathscr{X}}=\frac{c_2(\F)}{\mathrm{rk}(\F)}+\left(1-\frac{\deg c_2(\F)}{\mathrm{rk}(\F)} \right) o_S
$$
where $\F$ is any locally free sheaf of positive rank that lies on some constant cycle subvariety of some moduli space of stable objects in $D^b(\mathscr{X})$.
It was shown to be independent of the choice of constant cycle subvariety and moduli space. 
The authors of \emph{loc. cit.} conjecture that $o_{\mathscr{X}}=o_S$.
The definition of the Shen--Yin--Zhao filtration needs to be changed slightly to take the twist into account.

\begin{remark}
    The conjecture that $o_{\mathscr{X}}=o_S$ is supported in \cite{twistedSYZ} by descent arguments and the Bloch--Beilinson conjecture, but it is also implied by the Franchetta conjecture, see \intref{sec:franchetta}{Section} below.
\end{remark}

\begin{definition}[\cite{twistedSYZ}]
    Let $\mathscr{X}$ be a twisted $K3$ surface and $\mathcal{M}_\sigma(\mathscr{X},v)$ a moduli space of stable objects in $D^{(1)}(\mathscr{X})$.
    Define the Shen--Yin--Zhao filtration as
    $$
    S_i^{\text{SYZ}}\chow_0(\mathcal{M}_\sigma(\mathscr{X},v))\coloneq\left\langle [\E]~|~\E\in\mathcal{M}_\sigma(\mathscr{X},v) \text{ and }c_2(\E)-\mathrm{rk}(\E)o_{\mathscr{X}}\in S_i^{\text{OG}}\chow_0(S) \right\rangle
    $$
    Note that it agrees with \intref{def:untwisted-syz}{Definition} in the untwisted case.
\end{definition}

It was shown that
$c_2(\E)-\mathrm{rk}(\E)o_{\mathscr{X}}\in S_{d(\E)}^{\text{OG}}\chow_0(S)$ for all $\E\in\mathcal{M}_\sigma(\mathscr{X},v)$ \cite{twistedSYZ}, implying that it is exhaustive.
It was also shown to be independent of the modular interpretation.
Marian and Zhao's result also holds in the twisted case and it was proven in \emph{loc. cit.} that Voisin's filtration and the Shen--Yin--Zhao filtration agree.
We introduce the twisted incidence correspondence and extend their results to include Vial's co-radical filtration.

\begin{definition}[\cite{twistedSYZ}]\label{def:twisted-correspondence}
    Consider the set 
    $$
    \Gamma'\coloneq \left\langle (\xi, \E)~\middle|~ c_2(\E)-\mathrm{rk}(\E)o_{\mathscr{X}}=[\xi]+ko_S\text{ in }\chow_0(S) \right\rangle \subset S^{[n]}\times\mathcal{M}_\sigma(\mathscr{X},v),
    $$
    where $k$ is determined by the Mukai vector.
    There is a subvariety $\Gamma\subset \Gamma'$ such that both projections from it are dominant and generically finite, which we call the \emph{twisted incidence correspondence}.
\end{definition}

Indeed, $\Gamma'$ is a countable union of subvarieties and the second projection is surjective, since $c_2(\E)-\mathrm{rk}(\E)\in S_{d(\E)}^{\text{OG}}\chow_0(S)$.
By a Baire category argument, there is a component~$\Gamma_0\subset \Gamma'$ on which the second projection is dominant. 
Then, it follows from \cite[Lem. 1.1]{voisin16} that the symplectic forms on $S^{[n]}$ and $\mathcal{M}_\sigma(\mathscr{X},v)$ agree after pullback to~$\Gamma_0$.
In particular, since the second projection is dominant, the pullbacks must be of rank~$2d(v)$, implying that the first projection must also be dominant.
After possibly taking hyperplane sections of $\Gamma_0$ we obtain a subvariety $\Gamma$ such that both projections from it are dominant and generically finite.

The twisted incidence correspondence is used in \cite{twistedSYZ} to show that the birational motives of $\mathcal{M}_\sigma(\mathscr{X},v)$ and $S^{[n]}$ are isomorphic as co-algebra objects. 
As we will need the explicit isomorphism later, we give the proof here again.

\begin{proposition}[{\cite[Prop. 6.2]{twistedSYZ}}]\label{prop:iso-birat-motives}
    The twisted incidence correspondence $\Gamma$ induces an isomorphism 
    $$\mathfrak{h}^\circ(\mathcal{M}_\sigma(\mathscr{X},v))\cong\mathfrak{h}^\circ(S^{[n]})$$ 
    as co-algebra objects. 
    Furthermore, $\mathcal{M}_\sigma(\mathscr{X},v)$ admits a co-multiplicative birational Chow--Künneth decomposition in the sense of \intref{def:birational-CK}{Definition}.
\end{proposition}
\begin{proof}
    Since we are working over $\IC$, we may use \intref{lemma:co-algebra-iso}{Lemma} without having to base change to a universal domain. 
    Marian and Zhao's result \cite{marian-zhao}, which states that~$c_2(\E)=c_2(\F)$ if and only if $[\E]=[\F]$, ensure that the statement can be applied to show that $\mathfrak{h}^\circ(\mathcal{M}_\sigma(\mathscr{X},v))\cong\mathfrak{h}^\circ(S^{[n]})$ as co-algebra objects.
    It is proven in \cite[Thm. 4.5]{vial22} that there is a co-multiplicative birational Chow--Künneth decomposition $\{\varpi_{2i}\}_{0\leq i\leq n}$ on $S^{[n]}$. 
    This induces a birational Chow--Künneth decomposition on $\mathcal{M}_\sigma(\mathscr{X},v)$ via the isomorphism~$\mathfrak{h}^\circ(\mathcal{M}_\sigma(\mathscr{X},v))\cong~\mathfrak{h}^\circ(S^{[n]})$.
\end{proof}

We can now show that Vial's co-radical filtration agrees with Voisin's and the Shen--Yin--Zhao filtration also in the case of $\mathcal{M}_\sigma(\mathscr{X},v)$ and that \intref{prop:untwisted-criterion}{Proposition} also holds in the twisted case.

\begin{theorem}\label{prop:twisted-filtrations}
    Let $\mathscr{X}=(S,\beta)$ be a twisted $K3$ surface and $\mathcal{M}_\sigma(\mathscr{X},v)$ a moduli space of stable objects in $D^{(1)}(\mathscr{X})$. 
    Then Voisin's, the Shen--Yin--Zhao, and Vial's co-radical filtration agree,
    $$
    S_i^{\text{V}}\chow_0(\mathcal{M}_\sigma(\mathscr{X},v))=S_i^{\text{SYZ}}\chow_0(\mathcal{M}_\sigma(\mathscr{X},v))=R_i\chow_0(\mathcal{M}_\sigma(\mathscr{X},v)).
    $$
\end{theorem}
\begin{proof}
    The first equality is already known by \cite[Thm. 1.4]{twistedSYZ}. 
    Furthermore, it is known from \cite[Thm. 3]{vial22} that $R_i$ is of motivic nature. 
    More precisely,
    $$
    R_k\chow_0(T)=\left(\bigoplus_{i\leq k}\varpi_{2i}^T\right) \!{}_\ast\chow_0(T),
    $$
    where $\{\varpi^T_{2i}\}_{0\leq i\leq n}$ denotes the co-multiplicative birational Chow--Künneth decomposition on $T=S^{[n]}$ or $T=\mathcal{M}_\sigma(\mathscr{X},v)$ from \intref{prop:iso-birat-motives}{Proposition}.
    Since they are compatible, the isomorphism $\mathfrak{h}^\circ(\mathcal{M}_\sigma(\mathscr{X},v))\cong\mathfrak{h}^\circ(S^{[n]})$ sends $R_k\chow_0(S^{[n]})$ to $R_k\chow_0(\mathcal{M}_\sigma(\mathscr{X},v))$.
    The twisted incidence correspondence also respects $S_\bullet^{\text{SYZ}}$ by construction and $R_\bullet=~S_\bullet^{\text{SYZ}}$ on $S^{[n]}$ by \cite[Thm. 7.3]{vial22}.
    Hence, $S_i^{\text{SYZ}}\chow_0(\mathcal{M}_\sigma(\mathscr{X},v))=R_i\chow_0(\mathcal{M}_\sigma(\mathscr{X},v))$.
\end{proof}

\begin{proposition}\label{prop:criterion1}
    Let $\mathscr{X}=(S,\beta)$ be a twisted $K3$ surface and $\mathcal{M}_\sigma(\mathscr{X},v)$ a moduli space of stable objects in $D^{(1)}(\mathscr{X})$ of dimension $2n$.
    Let $\E_1,\ldots,\E_m,\F_i,\ldots,\F_m$ be points in $\mathcal{M}_\sigma(\mathscr{X},v)$.
    Then
    \begin{multline*}
    \sum_{i=1}^m[\E_i]=\sum_{i=1}^m[\F_i]\text{ in }\chow_0(\mathcal{M}_\sigma(\mathscr{X},v))\\
    \iff \sum_{i=1}^m c_2(\E_i)^{\times k}=\sum_{i=1}^m c_2(\F_i)^{\times k} \text{ in }\chow_0(S^k) \text{ for all }k\leq\min(m,n).
    \end{multline*}
\end{proposition}
\begin{proof}
    Recall that we have an isomorphism $\mathfrak{h}^\circ(\mathcal{M}_\sigma(\mathscr{X},v))\cong\mathfrak{h}^\circ(S^{[n]})$ as co-algebra objects by \intref{prop:iso-birat-motives}{Proposition}. 
    In particular, $\mathcal{M}_\sigma(\mathscr{X},v)$ satisfies the conclusion of \cite[Conj. 5.4]{vial22}.
    We know by \cite[Thm. 5.5(i)]{vial22} that 
    $$
    \mathfrak{h}^\circ(\mathcal{M}_\sigma(\mathscr{X},v))\cong \sym^{\leq n}\mathfrak{h}^\circ_\varpi(\mathcal{M}_\sigma(\mathscr{X},v)),
    $$
    for a birational idempotent correspondence $\varpi$ that factors as 
    $$
    \varpi\colon \mathfrak{h}^\circ(\mathcal{M}_\sigma(\mathscr{X},v))\twoheadrightarrow \mathfrak{h}^\circ_2(S)\hookrightarrow\mathfrak{h}^\circ(\mathcal{M}_\sigma(\mathscr{X},v)).
    $$
    The left arrow is given by $[\E]\mapsto c_2(\E)-\mathrm{rk}(\E)o_{\mathscr{X}}-co_S$, where $c=\deg c_2(\E)-\mathrm{rk}(\E)$ depends only on the Mukai vector.
    
    Finally, \cite[Prop. 2.1]{martin-vial} states that for a smooth projective variety $X$ with a direct summand $\mathfrak{h}_\varpi^\circ(X)$ that co-generates $\mathfrak{h}^\circ(X)$ in degree $\leq r$, then
    $$
    \sum_{i=1}^m[x_i]=\sum_{i=1}^m[y_i] \text{ in }\chow_0(X)\iff \sum_{i=1}^m(\varpi_\ast[x_i])^{\times k}=\sum_{i=1}^m(\varpi_\ast[y_i])^{\times k}\text{ in }\chow_0(X^k)
    $$
    for all $k\leq \min(m,r)$, as the birational motive of $\mathcal{M}_\sigma(\mathscr{X},v)$ is co-generated in degree~$n$.
    As the correspondence $\varpi$ factors through $\mathfrak{h}^\circ(S)$, we obtain that $\sum_{i=1}^m[\E_i]=~\sum_{i=1}^m[\F_i]$ in $\chow_0(\mathcal{M}_\sigma(\mathscr{X},v))$ if and only if
    $$
    \sum_{i=1}^m(c_2(\E)-ro_{\mathscr{X}}-co_S)^{\times k}=\sum_{i=1}^m(c_2(\F_i)-ro_{\mathscr{X}}-co_S)^{\times k}
    $$
    for all $k\leq\min(m,n)$.
    It follows by induction on $k$ that this is equivalent to the statement in the theorem.
\end{proof}

\section{Double EPW quartics as moduli spaces of twisted sheaves}\label{sec:quartics}

Introduced by Iliev, Kapustka, Kapustka, and Ranestad in \cite{ikkr17}, double EPW quartics $Y$ are hyperkähler fourfolds of $K3^{[2]}$-type that form a $19$-dimensional family. 
They carry a polarisation $h$ of degree $4$ with respect to the Beauville--Bogomolov form~$q$, i.e. $q(h,h)=4$, and an anti-symplectic involution with invariant lattice isometric to~$U(2)$.
The general double EPW quartic is isomorphic to a moduli space of stable objects on a twisted $K3$ surface of genus $2$.
Furthermore, they admit two Lagrangian fibrations. 

There are three different constructions of the general double EPW quartic $Y$, which we summarise in the following theorem. 
We then discuss the construction as moduli spaces in more detail.
We refer to \cite[Sec. 2]{dEPW-quartics-mazzanti} for more details on the other constructions.

\begin{theorem}[\cite{ikkr17}]\label{thm:construction}
    There is a $19$-dimensional family $\mathcal{U}$ of hyperkähler varieties~$Y$ of $K3^{[2]}$-type, called double EPW quartics, that admit a polarisation of Beauville degree~$4$ and an anti-symplectic involution $\iota\colon Y\to Y$ with invariant lattice $H^2(Y,\IZ)^\iota$ isometric to $U(2)$.
    Each double EPW quartic admits two Lagrangian fibrations $Y\to\IP^2$.
    The general double EPW quartic $Y$ admits three distinct realisations:
    \begin{enumerate}[label=(\roman*)]
        \item as the moduli space of stable sheaves twisted by a $2$-torsion Brauer class on a polarised $K3$ surface of genus $2$. The general double EPW quartic admits exactly two such associated $K3$ surfaces.
        \item\label{b} as the base of a $\IP^1$-bundle $\alpha\colon F(X)\to Y$, where $F(X)$ is the Hilbert scheme of $(1,1)$-conics in a Verra fourfold $X=Q\cap C(\IP^2\times\IP^2)$, where $Q$ is a general quadric hypersurface.
        \item\label{c} as a double cover $Y\to D_1^{\overline{A}}$ of a Lagrangian degeneracy locus associated to a unique Lagrangian space $A$. The base $D_1^{\overline{A}}\subset C(\IP^2\times\IP^2)$ is a quartic section in a cone. The projections induce the two Lagrangian fibrations $Y\to\IP^2$.
    \end{enumerate}
    Only construction \ref{c} applies to all double EPW quartics.
\end{theorem}

We briefly summarise the consequences of the second construction. 
It is shown in \cite[Sec. 3]{ikkr17} that there is a $\IP^1$-fibration $\alpha\colon F(X)\to Y$, where $F(X)$ is the Hilbert scheme of $(1,1)$-conics in $X$, i.e. those that project to $(1,1)$-curves under the double cover $\pi\colon X\to\IP^2\times\IP^2$.
Then, we have the following result, which will be needed later.

\begin{theorem}[{\cite[Thm. 1.5]{dEPW-quartics-mazzanti}}]\label{thm:chow-conics}
    Let $Y$ be a general double EPW quartic, $X$ its associated Verra fourfold. 
    The universal conic $P$ induces an isomorphism 
    $$\chow_0(Y)^-\xrightarrow{\cong}\chow_1(X)_{\mathrm{hom}}.$$
    The composition $P^t_\ast\circ P_\ast$ is also an isomorphism with inverse given up to scalars by multiplication by the square of the polarisation $h^2$.
\end{theorem}

\subsection{Construction of double EPW quartics as moduli spaces}\label{sec:construction}

Double EPW quartics were first realised as moduli spaces of twisted sheaves on $K3$ surfaces in \cite[Sec. 4]{ikkr17} via lattice-theoretic methods. 
Instead of studying this construction, we follow the more explicit approach in \cite{CamereKapustkaMongardi2018}.

Consider the Segre embedding $\IP^2\times\IP^2\hookrightarrow \IP^8$ and the projective cone $C(\IP^2\times\IP^2)\subset\IP^9$ over it. 
A \emph{Verra fourfold} $X=Q\cap C(\IP^2\times\IP^2)$ is the intersection of the cone with a general quadric $Q$.
The two projections $\pi_i\colon X\to \IP^2$ are quadric fibrations with discriminant locus $C_i\subset \IP^2$.
The double covers $S_i\to\IP^2$ branched along the $C_i$ are $K3$ surfaces and the quadric bundles $\pi_i\colon X\to\IP^2$ induce two $\IP^1$-fibrations on $S_i$, defining two Brauer classes $\beta_i\in\mathrm{Br}(S_i)$ of order $2$.
We consider only $S_1$ and $\pi_1$ for the remainder of this section, as the case of $S_2$ is analogous.
Denote $\pi^\ast\Ox_{\IP^2}(a,b)$ by $\Ox_X(a,b)$.

It is known by \cite[Thm. 4.2]{kuznetsov} that there is a semi-orthogonal decomposition
$$ 
D^b(X)=\langle\Phi(B^b(\IP^2,\mathcal{B}_0)),\Ox_X(-1,0),\Ox_X,\Ox_X(1,0),\Ox_X(0,1),\Ox_X(1,1),\Ox_X(2,1)\rangle,
$$
where $\Phi:D^b(\IP^2,\mathcal{B}_0)\to D^b(X)$ is a fully faithful functor and $\mathcal{B}_0$ is the sheaf of even parts of the Clifford algebra associated to $\pi_1$, as described in \cite{kuznetsov}.
Furthermore, it is proven in \cite[Lem. 4.2]{kuznetsov2} that the double cover $f\colon S_1\to\IP^2$ induces an equivalence 
$$f_\ast\colon D^b(S_1,\beta_1)\to D^b(\IP^2,\mathcal{B}_0).$$
Let $\Gamma$ be the inverse functor and let $\pr\colon D^b(X)\to\Phi(D^b(\IP^2,\mathcal{B}_0))$ be the projection functor. 

Let $F(X)$ be the Hilbert scheme of conics that project to $(1,1)$-curves under the double cover $\pi\colon X\to\IP^2\times\IP^2$. 
We need the existence of the $\IP^1$-fibration $\alpha\colon F(X)\to Y$ from point \ref{b} of \intref{thm:construction}{Theorem}. Define a map to the moduli space of stable objects with Mukai vector $(0,H,0)$, where $H$ is the polarisation of $S_1$,
$$F(X)\to \mathcal{M}_{(0,H,0)}(S_1,\beta_1),$$
by composing $\Gamma$ with 
$$ c\mapsto \pr(\mathcal{I}_{c/X}(1)). $$
It was shown in \cite[Thm. 5.1]{CamereKapustkaMongardi2018} that this map is well-defined, factors through the map~$\alpha\colon~F(X)\to Y$, and induces an isomorphism
$$ Y\cong \mathcal{M}_{(0,H,0)}(S_1,\beta_1). $$
Furthermore, this can be done with both $K3$ surfaces $S_1$ and $S_2$ associated to $X$, and hence we also have $Y\cong\mathcal{M}_{(0,H',0)}(S_2,\beta_2)$.

\subsection{The Franchetta property}\label{sec:franchetta}

In this section we briefly discuss the Francehtta property and show that it holds for cubes of Verra fourfolds, as we will need it later.
We refer to \cite[Sec. 3]{dEPW-quartics-mazzanti} for a more detailed discussion.

Let $\mathcal{X}\to B$ be a family of smooth varieties over a smooth base $B$.
We say that a cycle $\gamma\in\chow^\ast(\mathcal{X}_b)$ is \emph{generically defined} if it lies in the image of the pullback map from~$\mathcal{X}$.
We denote the subring of generically defined cycle over $B$ by $\gdch^\ast_B(\mathcal{X}_b)$. 
Note that generically defined cycles depend on the family $\mathcal{X}\to B$.
Which family is meant will usually be clear from the context, so we often omit the subscript $B$.

\begin{definition}\label{def:franchetta-property}
    Let $\mathcal{X}\to B$ be a family of smooth varieties over a smooth base $B$.
    We say that the family $\mathcal{X}\to B$ satisfies the \emph{Franchetta property} if for any closed point, or equivalently for the very general point, $b\in B$ we have that the cycle class map restricted to the subring of generically defined cycles,
    $$\mathrm{cl}\colon\gdch_B^\ast(\mathcal{X}_b)\to H^\ast(\mathcal{X}_b,\IQ),$$
    is injective.
\end{definition}

O'Grady noticed in \cite{og2} that the Beauville--Voisin class is generically defined and asked whether all generically defined zero-cycles on a $K3$ surface are multiples of it.
It was conjectured in \cite{flv19} that locally complete families of hyperkähler varieties satisfy the Franchetta property, but it is also known to hold in the case of some Fano varieties \cite{flv21}.

\begin{theorem}[{\cite[Thm. 3.5]{dEPW-quartics-mazzanti}}]\label{thm:franchetta}
    Let $\mathcal{X}^{n/\mathcal{U}}\to\mathcal{U}$ be the family of self-products of Verra fourfolds. 
    It satisfies the Franchetta property for $n\leq 3$, meaning that for any point $b\in\mathcal{U}$, the cycle class map restricted to generically defined cycles, 
    $$ \mathrm{cl}\colon \gdch^\ast(\mathcal{X}^n_b)\to H^\ast(\mathcal{X}_b^n,\IQ)$$
    is injective for $n\leq 3$.
    Furthermore, the family of double EPW quartics $\mathcal{Y}\to\mathcal{U}$ satisfies the Franchetta property.
    In particular, there exists a generically defined zero-cycle $o_Y$ on $Y$.
\end{theorem}
\begin{proof}
    It is known from \cite{laterveer-verra} that there is an isomorphism of Chow motives $$\mathfrak{h}(X)\cong\mathfrak{h}(S)(1)\oplus\bigoplus_{i=0}^4\mathds{1}(i),$$ where $S$ is a $K3$ surface associated to the Verra fourfold $X$, as in \intref{sec:construction}{Section}. 
    This construction can be done in families and by \cite[Rmk. 2.6]{flv19} it suffices to know that the Franchetta property holds for the cube of $S$, which is known from \cite[Thm.~1.5]{flv19}.

    Since $Y\cong\mathcal{M}_{(0,H,0)}(S,\beta)$, the Chow motive $\mathfrak{h}(Y)$ is a direct summand of $\bigoplus_{i=1}^r\mathfrak{h}(S^2)(l_i)$, where $r\in\IN$ and $l_i\in\IZ$, by \cite[Thm. 1.5]{flv21}. 
    Then, applying the Franchetta property for squares of genus-$2$ $K3$ surfaces \cite[Thm. 1.5]{flv19} yields the Franchetta property for $\mathcal{Y}\to\mathcal{U}$.
\end{proof}

\subsection{The twisted and untwisted Beauville--Voisin classes}

The aim of this subsection is to discuss how the Franchetta conjecture would imply that the twisted and untwisted Beauville--Voisin classes agree. 
We will use this idea to prove that $o_{\mathscr{X}}=o_S$ for the family of twisted $K3$ surfaces associated to double EPW quartics. 

We begin by discussing the situation in general and explain what is missing. 
Then, we discuss that this is known in the case of $K3$ surfaces associated to double EPW quartics.
Unfortunately, there remains a problem, for which we discuss a solution before proving the main theorem of this section.

Recall that 
$$ o_{\mathscr{X}}=\frac{c_2(\F)}{\mathrm{rk}(\F)}+\left(1-\frac{\deg c_2(\F)}{\mathrm{rk}(\F)} \right) o_S, $$
where $\F$ is any locally free sheaf of positive rank lying on some constant cycle subvariety of some moduli space of stable objects on $\mathscr{X}$.

Since Chow motives of moduli spaces on twisted $K3$ surfaces are direct summands of the Chow motives of powers of the underlying $K3$ surface \cite[Thm. 1.5]{flv21}, we expect them to satisfy the Franchetta property \cite[Conj. 2]{bvf}. 
Furthermore, we expect the classes of points on constant cycle subvarieties to be generically defined. 
Hence, it would follow that for any $\mathcal{F}$ on a constant cycle subvariety, $c_2(\F)$ is also generically defined, i.e. a multiple of $o_S$.
This would imply $o_{\mathscr{X}}=o_S$.

\begin{remark}
Unfortunately, this approach does not yield a proof of $o_{\mathscr{X}}=o_S$, even in the cases where the Franchetta property is known for powers of $K3$ surfaces, since it is not clear whether the classes of points lying on constant cycle subvarieties in some~$\mathcal{M}_\sigma(\mathscr{X},v)$ are generically defined.
\end{remark}

In the case of double EPW quartics, we can try to make this approach work.
We know that the fixed locus $Z\subset Y$ of the involution $\iota$ is a generically defined constant cycle subvariety \cite[Thm. 5.3]{dEPW-quartics-mazzanti}.
Furthermore, points on other constant cycle subvariety define the same class $o_Y$ as those on $Z$ \cite[Prop. 5.4]{dEPW-quartics-mazzanti}.
This removes most of the obstacles -- the only remaining problem now is that $Y$ is a moduli space of rank $0$ sheaves on $(S,\beta)$, and hence provides no information on $o_{\mathscr{X}}$. 

The solution to this is to construct an isomorphism between $Y$ and a moduli space of twisted sheaves of positive rank. 
A similar argument was also used in \cite[Sec. 3]{arbarello-sacca}.
We adapt it to the twisted case.

\begin{construction}\label{con_1}
We begin by reviewing the argument in the untwisted case.
Let $S$ be a $K3$ surface and $\F$ a sheaf of pure dimension $1$ on $S$. 
After possibly twisting by $\Ox_S(n)$, we may assume that $H^1(S,\F)=0$ and $\F$ is globally generated. 
Define $M_\F$ as the kernel of the evaluation map.
$$0\to M_\F\to H^0(S,\F)\otimes\Ox_S\to \F\to 0.$$
Then, one can compute the cohomology of $M_\F$ and $M_\F^\vee$, the latter of which is a locally free sheaf.
In particular, one can check that there are isomorphisms of differential graded Lie algebras $$\Ext^\ast(\F,\F)\cong\Ext^\ast(M_\F,M_\F)\cong \Ext^\ast(M_\F^\vee, M_\F^\vee).$$
In particular, one obtains isomorphisms between the deformation spaces of $\F$ and $M_\F$. 
Furthermore, it is proven that if $\F$ is stable, then so is  $M_\F$ \cite[Thm. 3.8]{arbarello-sacca}, yielding an isomorphism of moduli spaces. 
The key point of this construction for the following theorem is that $M_\F$ is of positive rank. 
This construction relies crucially on the fact that $\Ox_S$ is a spherical object in $D^b(S)$.

Consider now a twisted $K3$ surface $\mathscr{X}=(S,\alpha)$.
One can define a replacement for~$\Ox_S$ on $\mathscr{X}$, which we denote by $\L$, in the following manner. 
Considering $\mathscr{X}\to S$ as a $\mu_n$-gerbe, choose an étale cover $U\to S$ over which $\mathscr{X}$ becomes trivial, meaning that~$\mathscr{X}\times_S U\cong~B\mu_n\times U$.
Then, the tautological character $\mu_n\to \mathbb{G}_m$ defines a line bundle~$\Ox_{\mathrm{taut}}$ on $B\mu_n$. 
After pulling it back to $B\mu_n\times U$, this descends to a line bundle~$\L$ on $\mathscr{X}$.
If one prefers to work with \v{C}ech cocycles, one can glue $\Ox_{U_{ij}}$ along transition maps $g_{ij}$ such that $g_{ij}g_{jk}g_{ki}=\alpha_{ijk}$, where $\{\alpha_{ijk}\}$ is the \v{C}ech $2$-cocycle representing $\alpha$.
This yields an $\alpha$-twisted line bundle $\L$. 
Since this is a line bundle, we have $\mathcal{E}nd(\L)=\Ox_S$, and $\L$ is a spherical object, i.e. $\Ext^\ast(\L,\L)=\IC\oplus\IC[2]$. 
We refer to \cite{caldararu, lieblich} for more details.

Now, let $\F$ be an object in $D^b(S,\alpha)$ of rank $0$.
After twisting by $\Ox_S(n)$ for sufficiently large $n$, the untwisted sheaf $\F\otimes\L^\vee$ is globally generated and $H^1(S,\F)=0$. 
We define the spherical twist $T_\L(\F)$ of $\F$ via the triangle
$$R\Hom(\L,\F)\otimes \L\to \F\to T_\F(\F).$$
Since $\L$ is spherical, one obtains an isomorphism $\Ext^\ast(\F,\F)\cong\Ext^\ast(T_\L(\F),T_\L(\F))$ as differential graded algebras and of their deformation spaces. 
Furthermore, one obtains isomorphisms $\mathcal{M}_\sigma(\mathscr{X}, v)\cong \mathcal{M}_{(T_\L)_\ast\sigma}(\mathscr{X},(T_\L)_\ast v)$, where $(T_\L)_\ast v$ has positive rank.
Notice that this relies only on $\L$ being spherical, meaning it could be replaced by another spherical object.
\end{construction}

Since the rank of $(T_{\mathcal{L}})_\ast(\F)$ is of positive rank, we can apply this construction in the case of double EPW quartics to prove that the twisted and untwisted Beauville--Voisin classes agree.

\begin{theorem}\label{thm:BV-classes}
    Let $S$ be a general $K3$ surface of genus $2$ and $\beta\in\mathrm{Br}(S)$ the Brauer class of order $2$ discussed in \cite[Sec. 9.8]{van-geemen}.
    Denote the associated twisted by $K3$ surface by $\mathscr{X}$. 
    It is associated to a general double EPW quartic.
    Then the twisted and untwisted Beauville--Voisin classes agree, i.e. $$o_S=o_{\mathscr{X}}.$$
\end{theorem}
\begin{proof}
    Using the \Cref{con_1} described above, we obtain an isomorphism $$Y\cong\mathcal{M}_\sigma(\mathscr{X},v)$$
    for a Mukai vector $v$ of positive rank.
    Furthermore, by twisting with $\Ox_S(n)$ before applying $T_\L$, we can ensure that the rank of the Mukai vector is large enough to ensure that $r>\frac{v^2+2}{2}$.
    In that case, any $\mu_H$-stable sheaf is locally free. 
    Since the moduli space of $\mu_H$-stable sheaves $\mathcal{M}_H(\mathscr{X},v)$ is birational to $Y\cong\mathcal{M}_\sigma(\mathscr{X},v)$, we know that 
    $$\chow_0(\mathcal{M}_H(\mathscr{X},v))\cong \chow_0(\mathcal{M}_\sigma(\mathscr{X},v))$$
    and $\mathcal{M}_H(\mathscr{X},v)$ admits a birational Lagrangian fibration. 
    Hence, it follows from \cite[Prop. 3.1]{twistedSYZ} that there is a dense collection of constant cycle subvarieties of $\mathcal{M}_H(\mathscr{X},v)$ and that points lying on them define the same class. 
    Hence, there exists a constant cycle subvariety $C$ in $\mathcal{M}_H(\mathscr{X},v)$ whose image in $Y$ is a constant cycle subvariety. 
    Since points on constant cycle subvarieties in $Y$ define the same generically defined class~$o_Y\in\chow_0(Y)$ by \cite[Prop. 5.4]{dEPW-quartics-mazzanti}, we have $[\F]=o_Y$ for any $\F\in C$.
    Finally, recall the definition of the twisted Beauville--Voisin class as
    $$o_{\mathscr{X}}=\frac{c_2(\F)}{\mathrm{rk}(\F)}+\left(1-\frac{\deg c_2(\F)}{\mathrm{rk}(\F)} \right) o_S,$$
    Since $[\F]=o_Y$ is generically defined, so is $c_2(\F)$ and the only generically defined cycle on $S$ is $o_S$, implying that $o_{\mathscr{X}}=o_S$.
\end{proof}

\subsection{Applications}

Let $Y$ be a double EPW quartic.
It admits an involution $\iota$ and a generically defined Beauville--Voisin class $o_Y$, which is defined as the class of any point fixed by $\iota$. 
This is well-defined by \cite[Thm. 5.3]{dEPW-quartics-mazzanti}.
The involution $\iota$ induces an eigenspace decomposition
$$
\chow_0(Y)=\IQ o_Y\oplus \chow_0(Y)^-\oplus\chow_0(Y)_{\mathrm{hom}}^+.
$$
This induces a filtration
$$
\IQ o_Y\subset \IQ o_Y\oplus\chow_0(Y)^-\subset \chow_0(Y),
$$
which is known to agree with Voisin's filtration by \cite[Prop. 5.12]{dEPW-quartics-mazzanti}, and hence with Vial's co-radical and the Shen--Yin--Zhao filtrations.
We provide another proof of this via the construction as moduli spaces.

The Chern character induces a map on $\chow_0(Y)$ via the construction as a moduli space,
$$
\chow_0(Y)\cong \chow_0(\mathcal{M}_{(0,H,0)}(S,\beta))\xrightarrow{\mathrm{ch}_2}\chow_0(S)
$$
This map is induced by a correspondence, namely the second Chern character of a twisted universal family, which always exists.
The Bloch--Beilinson conjecture predicts that $\mathrm{ch}_2(F^4\chow_0(Y))\subset F^4\chow_0(S)=0$, meaning that $\mathrm{ch}_2$ should be zero on $\chow_0(Y)^+_{\mathrm{hom}}$.
This follows from the explicit description of the isomorphism $Y\cong \mathcal{M}_{(0,H,0)}(S,\beta)$ and a key geometric property. 
Two conics $c$ and $c'$ with $\iota(\alpha(c))=\alpha(c')$ have the property that~$c\cup c'$ is a hyperplane section in~$D_{(L,M)}=~\pi^{-1}(L\times~M)$, as in the proof of \cite[Thm.~5.3]{dEPW-quartics-mazzanti}. 
As $D_{(L,M)}$ is the zero-section of a vector bundle on $X$, we can control $\mathcal{I}_{c/X}(1)$.

\begin{theorem}\label{thm:K3-iso}
    Let $Y$ be a general double EPW quartic and let $S$ be either of the two K3 surfaces associated to $Y$. 
    Then the map
    $$\chow_0(Y)\cong \chow_0(\mathcal{M}_{(0,H,0)}(S,\beta))\xrightarrow{\mathrm{ch_2}} \chow_0(S)$$
    is $0$ on $\chow_0(Y)^+_{\mathrm{hom}}$, maps $o_Y$ to $o_S$, and induces an isomorphism 
    $$\chow_0(Y)^-\cong \chow_0(S)_{\mathrm{hom}}.$$
    Its inverse is given by the correspondence $\mathrm{ch}_2^t$ composed with multiplication by $h^2$.
\end{theorem}
\begin{proof}
    First, notice that we do not need to worry about the Brauer class, since we are working with Chow groups with rational coefficients.
    Using \cite[Thm. 1.1]{twistedSYZ}, the surjectivity of the map on all of $\chow_0(Y)$ follows from the argument following \intref{def:twisted-correspondence}{Definition} or from the argument used in the proof of \cite[Prop. 1.3]{og2} 

    We prove the injectivity on $\chow_0(Y)^-$ before proving that $\mathrm{ch}_2$ vanishes on $\chow_0(Y)^+_{\mathrm{hom}}$.
    Recall from \intref{thm:construction}{Theorem} that there is a $\IP^1$-fibration $\alpha\colon F(X)\to Y$ from the Hilbert scheme of $(1,1)$ conics on the Verra fourfold $X$.
    Then the universal conic $P$ induces a morphism 
    $$\psi\colon\chow_1(X)_{\mathrm{hom}}\xrightarrow{P_\ast^t}\chow_2(Y)^-_{\mathrm{hom}}\xrightarrow{\cdot h^2}\chow_0(Y)^-\xrightarrow{\mathrm{ch}_2}\chow_0(S)_{\mathrm{hom}}.$$
    Since the first two maps are isomorphisms by \intref{thm:chow-conics}{Theorem}, it suffices to prove that the composition is injective. 
    We show that there is a morphism $\chow_0(S)\to\chow_1(X)_{\mathrm{hom}}$ induced by a generically defined correspondence such that the composition is the identity.
    To prove this, we show that this is the case in cohomology and then use the Franchetta property for $X\times X$ from \intref{thm:franchetta}{Theorem}.
    Consider the composition 
    $$
    \phi\colon\chow_0(S)_{\mathrm{hom}}\xrightarrow{\mathrm{ch}_2^t}\chow^2(Y)^-\xrightarrow{\cdot h^2}\chow_0(Y)^-\xrightarrow{P_\ast}\chow_1(X)_{\mathrm{hom}},
    $$
    and define the primitive motive $\mathfrak{h}_{\mathrm{pr}}^4(X)$ of $X$ as the motive cut out by the idempotent correspondence
    $$
    \pi_4^X=\Delta_X-X\times x-x\times X - g_1\times g_1 g_2^2 - g_2\times g_2 g_1^2 - \frac{g_1^2\times g_1^2}{\deg g_1^2 g_2^2}-\frac{g_2^2\times g_2^2}{\deg g_1^2 g_2^2},
    $$
    where the $g_i$ are the two pullbacks of the hyperplane classes along the two projections~$\pi_i\colon X\to \IP^2$ and $x\in X$ is a point.
    Let $\theta$ be the endomorphism of $\mathfrak
    {h}_{\mathrm{pr}}^4(X)(-1)$ defined by~$\theta=~\pi_4^X\circ\phi\circ\psi\circ \pi_4^X$.
    Since $H_{\mathrm{pr}}^k(X)=0$ for $k\neq 4$, the cohomological realisation becomes 
    $$
    H^4_{\mathrm{pr}}(X)\xrightarrow{P^t_\ast}H_{\mathrm{pr}}^2(Y)\xrightarrow{\cdot h^2}H_{\mathrm{pr}}^6(Y)\xrightarrow{\mathrm{ch}_2}H_{\mathrm{pr}}^2(S)\xrightarrow{\mathrm{ch}_2^t}H_{\mathrm{pr}}^2(Y)^-\xrightarrow{\cdot h^2}H_{\mathrm{pr}}^6(Y)\xrightarrow{P_\ast}H_{\mathrm{pr}}^4(X).
    $$
    Notice that for the very general $X$, $H_{\mathrm{pr}}^4(X)=H_{\mathrm{tr}}^4(X)$, which is a simple Hodge structure. 
    Hence, this endomorphism is a multiple of the identity for the very general $X$. 
    This holds for all $X$ after spreading out.
    The second Chern character of a universal family induces an isomorphism between $H^2(Y)$ and the orthogonal complement of a class of square $2$ in the Mukai lattice $\tilde{H}^2(S)$. 
    Furthermore, both $\mathrm{ch}_2$ and $\mathrm{ch}_2^t$ are isomorphisms on transcendental cohomology \cite[Sec. 3.3]{yoshioka}. 
    Multiplication by $h^2$ is an isomorphism by the Hard Lefschetz Theorem and $P_\ast$ and $P_\ast^t$ are isomorphisms by \cite[Prop. 5.9]{dEPW-quartics-mazzanti}.
    Hence, the composition is a non-zero multiple of the identity and we may apply the Franchetta property for $X\times X$ to deduce that this is also the case up to rational equivalence.

    It remains to prove that the map vanishes on $\chow_0(Y)^+_{\mathrm{hom}}$.
    First, consider $o_Y$. 
    As the isomorphism $Y\cong\mathcal{M}$ and the map $\mathrm{ch}_2$ are generically defined, it must be sent to a multiple of $o_S$.
    We know that the Mukai vector of any object in $\mathcal{M}_{(0,H,0)}(S,\beta)$ is~$(0,H,0)$, so $\mathrm{ch}_2(o_Y)=\frac{1}{2}H^2=o_S$.
    Next, consider a cycle of the form $y+\iota(y)$ for~$y\in Y$. 
    Let $c$ and~$c'$ be conics such that $\alpha(c)=y$ and $\alpha(c')=\iota(y)$.
    Denote the lines that $c$ and $c'$ project to under the two projections $\pi_i\colon X\to \IP^2$ by $L$ and $M$.
    The key geometric property that we will use is the fact that $c\cup c'$ is a hyperplane section in $D_{(L,M)}=\pi^{-1}(L\times M)$, as in the proof of \cite[Thm. 5.3]{dEPW-quartics-mazzanti}.
    For the sake of clarity we denote $D_{(L,M)}$ by $D$ for the remainder of the proof.
    From the short exact sequences
    $$0\to \mathcal{I}_{D/X}(1,1)\to\mathcal{I}_{c/X}(1,1)\to \mathcal{I}_{c/D}(1,1)\to 0$$
    and
    $$0\to \Ox_X\to \Ox_X(1,0)\oplus \Ox_X(0,1)\to \mathcal{I}_{D/X}(1,1)\to 0,$$
    we see that $\pr(\mathcal{I}_{c/X}(1,1))=\pr(\mathcal{I}_{c/D}(1,1))$.
    Let $H_2$ be the pullback of a hyperplane along the second projection. From the exact sequences 
    $$0\to \Ox_{H_2}(0,1)\to \Ox_{H_2}(1,1)\to \Ox_D(1,1)\to 0 $$
    and
    $$0\to\Ox_X(0,-1)\to\Ox_X\to\Ox_{H_2}\to 0$$ 
    we obtain that $\pr(\Ox_D(1,1))=0$.
    Finally, we using the fact that $c\cup c'$ is a hyperplane in $D$, we obtain
    $$\mathcal{I}_{c/D}\otimes\mathcal{I}_{c'/D}=\mathcal{I}_{c\cup c' /D}=\Ox_D(-1,-1).$$
    Twisting with $\Ox_X(2,2)$ and applying the Chern character to the above, we obtain 
    \begin{align*}
        \mathrm{ch}_2(\Ox_D(1,1))&=\mathrm{ch}_2(\mathcal{I}_{c/D}(1,1)\otimes\mathcal{I}_{c'/D}(1,1))\\
        &=\frac{1}{2}(c_1(\mathcal{I}_{c/D}(1,1)+c_1(\mathcal{I}_{c'/D}(1,1))))^2-c_2(\mathcal{I}_{c/D}(1,1))-c_2(\mathcal{I}_{c'/D}(1,1))\\
        &=\frac{1}{2}c_1(\Ox_D(1,1))^2-c_2(\mathcal{I}_{c/D}(1,1))-c_2(\mathcal{I}_{c'/D}(1,1)).
    \end{align*}
    Since $\pr(\Ox_D(1,1))=0$, we get that 
    $$c_2(\pr(\mathcal{I}_{c/X}(1,1)))+c_2(\pr(\mathcal{I}_{c'/X}(1,1)))=0,$$
    meaning that $\mathrm{ch}_2$ composed with the isomorphism $Y\cong\mathcal{M}_{(0,H,0)}(S,\beta)$ vanishes on cycles of the form $y+\iota(y)-2o_Y$, completing the proof.
\end{proof}

Let us now consider the filtrations from \intref{sec:untwisted}{Section} for double EPW quartics.
Recall that $o_{(S,\beta)}=o_S$ by \intref{thm:BV-classes}{Theorem}, so we may safely ignore the twisted Beauville--Voisin class.

Let $X$ be the Verra fourfold associated to $Y$.
We define a filtration on $\chow_1(X)$ analogous to the filtration defined by Shen--Yin in the setting of Fano varieties of lines on cubic fourfolds \cite{shen-yin}.
Recall that a \emph{uniruled divisor} in a $2n$-dimensional hyperkähler variety $T$ is a divisor $D\subset T$ that admits a rational map to a $(2n-2)$-dimensional variety~$B$ such that the fibres of $D\dashrightarrow B$ are rational curves.
It is known that uniruled divisors exist in hyperkähler varieties of $K3^{[2]}$-type \cite{cmp21}.
\begin{definition}
    Define a filtration by subsets $S^{\text{SY}}_i\chow_1(X)$ as the union of all classes of the form $$k\xi+[c_1]+\ldots +[c_i],$$ where $k\in\IQ$, $\xi$ is the class of a conic such that $\alpha(\xi)$ lies on a constant cycle surface, and~$c_j$ are conics such that $\alpha(c_j)$ lies on a fixed uniruled divisor in $Y$.
    Define the \emph{Shen--Yin filtration} as
    $$S^{\text{SY}}_i\chow_0(Y)\coloneq\left\langle\left\{\alpha_\ast[c]~\middle|~c\in F(X)\text{ and }[c]\in S^{\mathrm{SY}}_i\chow_1(X) \right\} \right\rangle.$$
\end{definition}
This is well-defined, as all points on constant cycle surfaces have the same class by \cite[Prop. 5.4]{dEPW-quartics-mazzanti}.
Furthermore, it is independent of the choice of uniruled divisor by \cite[Lem. 1.1]{shen-yin} and exhaustive by \cite[Cor.~4.4]{laterveer-verra}.

\begin{proposition}\label{prop:filtrations}
    Let $Y$ be a general double EPW quartic.
    Then, all four filtrations, the Shen--Yin, Voisin's, Shen--Yin--Zhao, and Vial's co-radical filtration, agree,
    $$
    S_i^{\text{SY}}\chow_0(Y)=S_i^{\text{V}}\chow_0(Y)=S_i^{\text{SYZ}}\chow_0(Y)=R_i\chow_0(Y).
    $$
    Furthermore, they all agree with the filtration obtained from the decomposition of $\chow_0(Y)$,
    $$
    \IQ o_Y\subset \IQ o_Y\oplus \chow_0(Y)^-\subset\chow_0(Y).
    $$
\end{proposition}
\begin{proof}
    We know from \intref{prop:twisted-filtrations}{Proposition} that the latter three filtrations all agree, so it suffices to prove the first equality.
    Consider the composition
    $$ \chow_1(X)\xrightarrow{P^t_\ast}\chow_2(Y)\xrightarrow{\cdot h^2}\chow_0(Y)\xrightarrow{\mathrm{ch}_2}\chow_0(S),$$
    which sends $\xi$ to $o_Y$ and induces an isomorphism between $\chow_1(X)_{\mathrm{hom}}$ and $\chow_0(S)_{\mathrm{hom}}$ by \intref{thm:chow-conics}{Theorem} and \intref{thm:K3-iso}{Theorem}.
    It is known that $S_1^\mathrm{V}\chow_0(Y)$ is controlled by uniruled divisors for varieties of $K3^{[2]}$-type varieties, meaning that it is exactly the group of zero-cycles generated by classes of points that lie on uniruled divisors \cite{cmp21}.
    Furthermore, we know that $S_1^{\mathrm{V}}\chow_0(Y)=S_1^{\mathrm{SYZ}}\chow_0(Y)$, meaning that under the isomorphism above, the filtrations $S_1^{\mathrm{OG}}\chow_0(S)$ and $S_1^{\mathrm{SY}}\chow_1(X)$ agree.
    Hence, as they are defined in the same way, $S_1^{\mathrm{SY}}\chow_0(Y)=S_1^{\mathrm{SYZ}}\chow_0(Y)$.

    For the second claim, recall from \intref{prop:iso-birat-motives}{Proposition} that there is a co-multiplicative birational Chow--Künneth decomposition on $Y$ and hence
    $$
    R_k\chow_0(Y)=\left(\bigoplus_{i\leq k}\varpi_{2i}\right) \!_\ast\chow_0(Y).
    $$
    The co-multiplicative birational Chow--Künneth on $S^{[2]}$ is explicitly constructed in the proof of \cite[Thm.~4.5]{vial22}.
    Keeping track of the isomorphism via the twisted incidence correspondence and combining it with \intref{thm:K3-iso}{Theorem}, we see that $\chow_0(\mathfrak{h}_0^\circ(Y))=\IQ o_Y$ and $\chow_0(\mathfrak{h}^\circ_2(Y))=\chow_0(Y)^-$.
    This means that $(\varpi_2)_\ast$ is the projection onto $\chow_0(Y)^-$, completing the proof.
\end{proof}

The criterion from \intref{prop:criterion1}{Proposition} also applies to double EPW quartics. 
For the case of individual points $(m=1)$, combining it with \intref{thm:K3-iso}{Theorem} shows that two points~$x,y\in Y$ define the same class in $\chow_0(Y)$ if and only if $[x]^-=[y]^-$ in $\chow_0(Y)^-$, where the superscript denotes the anti-invariant part.
Using \intref{thm:chow-conics}{Theorem}, this translates to a geometric criterion: $[x]=[y]$ in $\chow_0(Y)$ if and only if there exist conics $c,c'$ in the associated Verra fourfold $X$ such that $[c]=[c']$ in $\chow_1(X)$ and $\alpha_\ast[c]=[x]$ and~$\alpha_\ast[c']=[y]$.
We now extend this to higher-degree effective zero-cycles. 

\begin{theorem}\label{thm:criterion2}
    Let $Y$ be a general double EPW quartic with associated Verra fourfold~$X$ and recall the $\IP^1$-fibration $\alpha\colon F(X)\to Y$ from \intref{thm:construction}{Theorem}. 
    Let $x_i,y_i\in Y$ and $c_i, c_i'$ be~$(1,1)$-conics in $X$ satisfying $\alpha_\ast[c_i]=[x_i]$ and $\alpha_\ast[c_i']=[y_i]$ in $\chow_0(Y)$.
    Then
    \begin{equation*}
        \sum_{i=1}^m[x_i]=\sum_{i=1}^m[y_i]\text{ in }\chow_0(Y)\iff 
        \begin{cases}
            \sum_{i=1}^m [c_i]=\sum_{i=1}^m [c_i']\text{ in }\chow_1(X), \text{ and}\\
            \sum_{i=1}^m[c_i]\times[c_i]=\sum_{i=1}^m[c_i']\times [c_i']\text{ in }\chow_2(X\times X).
        \end{cases}
    \end{equation*}
\end{theorem}
\begin{proof}
    We have $\mathfrak{h}^\circ(Y)\cong\sym^{\leq 2}\mathfrak{h}^\circ_\varpi(Y)$ by \intref{prop:iso-birat-motives}{Proposition} and \cite[Thm. 5.5]{vial22}.
    The idempotent $\varpi$ factors as 
    $$
    \varpi\colon \mathfrak{h}^\circ(Y)\twoheadrightarrow \mathfrak{h}_2^\circ(S)\hookrightarrow \mathfrak{h}^\circ(Y).
    $$
    As in the proof of \intref{prop:filtrations}{Proposition}, it follows from \intref{thm:K3-iso}{Theorem} that this agrees with the projection onto $\chow_0(Y)^-$ after taking Chow groups. 
    Recall the correspondence~$P_\ast$ induced by the universal conic and $\alpha\colon F(X)\to Y$ from the discussion following \intref{thm:construction}{Theorem}, and consider the morphism of Chow motives
    $$
    \mathfrak{h}^6(Y)_{\mathrm{pr}}\xrightarrow{P_\ast}\mathfrak{h}^4(X)_{\mathrm{pr}}(-1)\xrightarrow{(\pr_1^\ast h^2)_\ast\circ P_\ast^t}\mathfrak{h}^6(Y)_{\mathrm{pr}}
    $$
    Let $\varpi'$ be the idempotent birational correspondence obtained by restricting this composition to the generic point. 
    Recall that $\mathfrak{h}^4(X)_{\mathrm{pr}}(-1)$ is the Chow motive cut out by the idempotent 
    $$\pi_4^X=\Delta_X-X\times x-x\times X - g_1\times g_1 g_2^2 - g_2\times g_2 g_1^2 - \frac{g_1^2\times g_1^2}{\deg g_1^2 g_2^2}-\frac{g_2^2\times g_2^2}{\deg g_1^2 g_2^2},$$
     where $g_1$ and $g_2$ are the two divisors pulled back to $X$ from $\IP^2$.
    Furthermore, we define~$\mathfrak{h}^6(Y)_{\mathrm{pr}}$ as the Chow motive cut out by $P_\ast\circ \pi_{\mathrm{pr}}^4\circ(\pr_1^\ast h^2)_\ast\circ P_\ast^t$. 
    It follows from \intref{thm:chow-conics}{Theorem} that $\varpi'$ induces the projection onto $\chow_0(Y)^-$ up to a scalar multiple.
    After rescaling, we obtain $\varpi=\varpi'$ by \intref{lemma:defined-by-chow}{Lemma}. 
    We also see that $\pi_{\mathrm{pr}}^4\circ P_\ast([y])=[c]-[c_0]$, where $c$ and $c_0$ are conics with $\alpha_\ast[c]=[y]$ and $\alpha_\ast[c_0]=o_Y$. 
    We conclude the proof using \cite[Prop.~2.1]{martin-vial}and rearranging, as in the case of \intref{prop:criterion1}{Proposition}. 
\end{proof}
\begin{remark}\label{rmk:elementary}
    While the case $m=1$ follows from \intref{thm:chow-conics}{Theorem}, it also admits a more elementary proof that is independent of \intref{thm:chow-conics}{Theorem}.
    This alternative proof requires only \cite[Lem. 5.1]{dEPW-quartics-mazzanti}, which is itself independent of the rest of \emph{loc. cit.}
    However, this method does not generalise to $m>1$.

    Given two conics $c,c'$ with $[c]=[c']$, applying $P_\ast^t$ means that $P_\ast^t(c)=P_\ast^t(c')$.
    Squaring and using \cite[Lem.~5.1]{dEPW-quartics-mazzanti} we obtain 
    ($P_\ast^t(c)=D_c$ in the notation of \emph{loc. cit.})
    $$\alpha_\ast(h_F\cdot\sum_{j=1}^8\ell_j)+3m[\iota(\alpha(c))]-m[\alpha(c)]=\alpha_\ast(h_F\cdot\sum_{j=1}^8\ell_j')+3m[\iota(\alpha(c'))]-m[\alpha(c')],$$
    where $\ell_j$ and $\ell_j'$ are the pencils of conics contained in the del Pezzo surfaces $D_{(L,M)}$ and $D_{(L',M')}$ associated to $c$ and $c'$, respectively. 
    We claim that there is an equality~$\alpha_\ast(h_F\cdot~\sum_{j=1}^8\ell_j)=\alpha_\ast(h_F\cdot\sum_{j=1}^8\ell_j')$.
    Indeed, the $8$ points in $Y$ associated to the $\ell_j$ are contracted to $4$ points in $D_1^{\overline{A}}$, which is a degree $4$ cover of $\IP^2\times\IP^2$ and these points are contracted to a point in $\IP^2\times\IP^2$, meaning that the two classes are pulled back from $\IP^2\times \IP^2$ and hence agree.
    Rearranging proves that $\alpha_\ast[c]=\alpha_\ast[c']$.
\end{remark}

\printbibliography

\end{document}